\documentclass[14pt]{amsart}
\usepackage{extsizes,comment,amsmath, amssymb}
\usepackage[margin=0.78in]{geometry}
\usepackage{pgfplots,mathtools,float}
\usepackage{microtype,array,subcaption}
\pgfplotsset{compat=newest}

\usepackage[colorlinks=false,linkcolor=blue,citecolor=green!65!black]{hyperref}
\usepackage[sort,compress]{cite}
\usepackage[capitalise,nameinlink]{cleveref}

\newcommand{\citep}{\cite}

\newcolumntype{M}[1]{>{\centering\arraybackslash}m{#1}}

\newtheorem{remark}{Remark}

\newtheorem{lemma}{Lemma}
\newtheorem{theorem}{Theorem}

\numberwithin{equation}{section}
\crefname{subsection}{Subsection}{Subsections}
\crefformat{equation}{(#2#1#3)}
\crefformat{align}{(#2#1#3)}
\crefformat{enumi}{(#2#1#3)}
\crefname{figure}{Figure}{Figures}

\newcommand{\com}{%
    \mathrel{\mathrlap{{\mspace{4mu}\lhook}}{\hookrightarrow}}}

\newcommand{\dd}{\textup{d}}
\newcommand{\dx}{\,\textup{d}x}

\newcommand{\ds}{\,\textup{d}s}

\newcommand{\ddt}{\frac{\dd}{\dd t}}

\newcommand{\R}{\mathbb{R}}

\newcommand{\p}{\partial}
\newcommand{\pt}{\partial_t}
\newcommand{\pta}{{}^\textup{C}\!\p_t^\alpha}

\newcommand{\Dta}{{}^\textup{RL}\!\p_t^\alpha}
\newcommand{\Dtb}{{}^\textup{RL}\p_t^{1-\alpha}}
\newcommand{\ga}{g_{\alpha}}
\newcommand{\gb}{g_{1-\alpha}}

\DeclareFontFamily{U}{bbold}{}
\DeclareFontShape{U}{bbold}{m}{n}
{
	<-5.5> s*[1.069] bbold5
	<5.5-6.5> s*[1.069] bbold6
	<6.5-7.5> s*[1.069] bbold7
	<7.5-8.5> s*[1.069] bbold8
	<8.5-9.5> s*[1.069] bbold9
	<9.5-11> s*[1.069] bbold10
	<11-15> s*[1.069] bbold12
	<15-> s*[1.069] bbold17
}{}

\def\signed #1{{\leavevmode\unskip\nobreak\hfil\penalty50\hskip2em
		\hbox{}\nobreak\hfil(#1)%
		\parfillskip=0pt \finalhyphendemerits=0 \endgraf}}

\newsavebox\mybox

\begin{document}
\title[A time-fractional Fisher--KPP equation for tumor growth]{A time-fractional Fisher--KPP equation for tumor growth: Analysis and numerical simulation}
\author[Marvin Fritz]{Marvin Fritz$^{1}$}
\author[Nikos Kavallaris]{Nikos Kavallaris$^{2}$}
\footnotetext[1]{Computational Methods for PDEs, Radon Institute for Computational and Applied Mathematics, Linz, Austria}
\footnotetext[2]{Department of Mathematics and Computer Science, Karlstad University, Sweden}

\begin{center}
    \bf \MakeUppercase{A time-fractional Fisher--KPP equation for tumor growth: Analysis and numerical simulation}
\end{center}

\begin{center}
   \sc Marvin Fritz$^{1}$,  Nikos I. Kavallaris$^{2}$
\end{center}

\begin{quote} \small
\textsc{Abstract.}
We study a time-fractional Fisher--KPP equation involving a Riemann--Liouville fractional derivative acting on the diffusion term, as derived by Angstmann and Henry (Entropy, 22:1035, 2020). The model captures memory effects in diffusive population dynamics and serves as a framework for tumor growth modeling.
We first establish local well-posedness of weak solutions. The analysis combines a Galerkin approximation with a refined a priori estimate based on a Bihari--Henry--Gronwall inequality, addressing the nonlinear coupling between the fractional diffusion and the reaction term. For small initial data, we further prove global well-posedness and asymptotic stability.
A numerical method based on a nonuniform convolution quadrature scheme is then proposed and validated. Simulations demonstrate distinct dynamical behaviors compared to conventional formulations, emphasizing the physical consistency of the present model in describing tumor progression.

    \smallskip

\end{quote}

\section{Introduction}
The Fisher--KPP equation is a fundamental model in the theory of reaction-diffusion equations, originally introduced to describe the spread of advantageous genes in a population \cite{fisher1937wave}. Beyond population dynamics, it has found applications in fields as diverse as ecology \cite{shigesada1997biological,simpson2024fisher} and tumor growth \cite{browning2019bayesian,everett2020tutorial}, providing a paradigmatic example of nonlinear spatiotemporal dynamics in complex biological systems.

In recent years, there has been growing interest in nonlocal extensions of reaction–diffusion equations \cite{metzler2000random,KavallarisSuzuki2018,KavallarisLatosSuzuki2023,del2013fractional,henry2006anomalous}, motivated by the need to describe anomalous transport phenomena, memory effects, and long-range interactions that classical local diffusion models fail to capture. Among these approaches, time-fractional derivatives have emerged as a powerful framework for modeling subdiffusive dynamics, where the mean squared displacement grows sublinearly in time due to mechanisms such as trapping, crowding, or viscoelastic damping \cite{klages2008anomalous,angstmann2020time,mainardi2010fractional}. Such effects are widely observed in biological tissues, porous and heterogeneous media, and complex fluids, where transport processes deviate markedly from Fickian behavior \cite{timsina2021mathematical,benson2000application,magin2006fractional}. The resulting nonlocal formulations thus provide a more faithful representation of diffusion–reaction processes in complex systems.

A key challenge in extending the Fisher–KPP equation to the fractional setting lies in identifying a physically consistent mathematical formulation. While many studies simply replace the classical time derivative with a Caputo derivative \cite{huy2024global,rahby2025asymptotic}, this approach does not always reflect the underlying stochastic mechanisms responsible for anomalous diffusion. As shown in \cite{angstmann2020time} and further discussed in \cite{fritz2024analysis,heinsalu2007use,weron2008modeling,angstmann2015generalized} for related Fokker--Planck models, a derivation based on continuous-time random walks and subordination arguments naturally leads to a formulation in which the Riemann–Liouville fractional derivative acts on the diffusion term. This distinction is not merely technical: as demonstrated in this work, the two formulations can exhibit markedly different qualitative behavior, both analytically and numerically.

The main contributions of this paper are twofold. First, we provide a rigorous analysis of the time-fractional Fisher–KPP equation with a Riemann–Liouville derivative in the diffusion term, establishing local and global well-posedness of weak solutions. The proof combines a Galerkin approximation with a priori estimates obtained via a Bihari–Henry–Gronwall inequality, extending techniques developed for nonlinear time-fractional PDEs \cite{fritz2021equivalence,fritz2021subdiffusive,fritz2021timefractional,fritz2024analysis,fritz2024well}. Second, we develop a nonuniform convolution quadrature scheme for numerical approximation and compare the dynamics of the physically consistent model with those of the conventional Caputo formulation. The results highlight the importance of model consistency for accurately capturing the intrinsic dynamics of subdiffusive reaction–diffusion systems, particularly in the context of tumor growth and progression.

The paper is organized as follows. In \cref{Sec:Modeling}, we review the derivation of the time-fractional Fisher–KPP equation and discuss the physical interpretation of different formulations. \cref{ap} introduces the main notation and presents several preliminary analytical results. \cref{Sec:Analysis} contains the main analytical theorems, while \cref{Sec:Numerics} reports numerical simulations that illustrate the qualitative differences between models. 

\section{Modeling} \label{Sec:Modeling}

The classical Fisher--KPP equation describes diffusive population dynamics or the evolution of tumor cell density with logistic growth. When incorporating memory effects through fractional derivatives, careful consideration must be given to the appropriate mathematical formulation. In population dynamics, one typically assumes a constant per capita birth rate and no mortality, which leads to logistic reaction kinetics. As derived in \cite[Eq.~(32)]{angstmann2020time} using subdiffusive processes with reaction terms, the physically consistent model involves the Riemann--Liouville derivative $\Dtb$ of order \(1-\alpha\), with \(\alpha \in (0,1)\) (see \cref{Def:RL} for the precise definition). This derivative acts on the diffusion term, yielding
\begin{equation} \label{Eq:OriginalFisher}
    \partial_t u = D \Delta \Dtb u + ru(1-u),
\end{equation}
where \(D>0\) is the diffusion coefficient and \(r>0\) denotes the intrinsic growth rate.  
The nonlinear term \(ru(1-u)\) represents logistic growth. In the context of tumor dynamics, regions fully occupied by tumor cells (\(u=1\)) or free of them (\(u=0\)) show no further proliferation, while intermediate concentrations (\(u \approx 0.5\)) correspond to maximal growth. More generally, the reaction term drives the system toward the carrying capacity (\(u=1\)).

Our analysis allows for either homogeneous Dirichlet or Neumann boundary conditions. Dirichlet conditions prescribe the population at the boundary—often zero, representing a hostile or cell-free region—while homogeneous Neumann conditions correspond to impermeable boundaries, ensuring no flux of cells across the domain boundary. Biologically, this models a closed tissue region in which cells diffuse and proliferate but cannot escape. In what follows, we adopt this framework on a bounded domain \( \Omega \subset \mathbb{R}^d \).

By convolving \cref{Eq:OriginalFisher} with the kernel \( \gb \) and using the relations \( \pta u=\gb*\pt u \) and \( \gb * \Dtb u = u \) (see \cref{Eq:InverseConvolution}), we obtain the equivalent formulation
\begin{equation} \label{Eq:Fisher}
    \pta u = D\Delta u + r\,\gb*(u-u^2).
\end{equation}
This form is more convenient for analysis since the Caputo derivative naturally incorporates the initial condition. We establish local existence of weak solutions for short times, although it remains difficult to guarantee \(u\in[0,1]\) when \(u_0\in[0,1]\), due to the convolution structure of the source term and the absence of a weak comparison principle. This issue is common in nonlocal PDEs, even for linear equations of the type \(\pt u = \Delta u + Ku\) with a nonlocal operator \(K\); see \cite[Chapter~V]{quittner2019superlinear} and \cite[Chapter~1]{KavallarisSuzuki2018}.

A common but not physically consistent formulation appearing in the literature (see, e.g., \cite{huy2024global,rahby2025asymptotic}) replaces the time derivative in \cref{Eq:OriginalFisher} with a Caputo derivative:
\begin{equation} \label{Eq:WrongModel}
    \pta u = D\Delta u + r(u-u^2).
\end{equation}
Although mathematically tractable and admitting a comparison principle, this equation does not correctly represent the underlying stochastic mechanisms, as emphasized in \cite{angstmann2020time}. Numerical experiments in \cref{Sec:Numerics} reveal clear differences between \cref{Eq:Fisher} and \cref{Eq:WrongModel}, particularly in tumor-growth scenarios. From a modeling standpoint, fractional derivatives arising from Langevin dynamics and subordination arguments are naturally associated with the diffusion operator rather than the reaction term. This interpretation is consistent with the time-fractional Fokker--Planck equation, where \cite{heinsalu2007use} showed that Caputo-type formulations are ``physically defeasible,'' meaning their solutions do not correspond to valid stochastic processes. Both models, \cref{Eq:Fisher} and \cref{Eq:WrongModel}, reduce to the classical Fisher--KPP equation in the limit \(\alpha = 1\). A brief remark in \cref{Sec:Analysis} indicates how global well-posedness for the Caputo-type model can be established via a weak comparison principle.

From a structural viewpoint, the classical Fisher--KPP equation (\(\alpha=1\)) possesses a gradient-flow structure:  
\[
\pt u = -\nabla_{L^2} E(u),
\quad
E(u)=\int_\Omega \frac{D}{2}|\nabla u|^2 - \frac{r}{2}u^2 + \frac{r}{3}u^3 \,\dx.
\]
The Caputo-based model \cref{Eq:WrongModel} can be interpreted as a \emph{time-fractional gradient flow} in the sense of \cite{fritz2021equivalence}, whereas the physically consistent model \cref{Eq:Fisher} cannot be directly cast in that framework.  

Another key structural property arises from \emph{mass balance}, obtained by testing the variational formulation with \(1\):
\[
\ddt \int_\Omega u(x)\,\dx = \int_\Omega r\,u(1-u)\,\dx,
\]
which holds for the classical Fisher--KPP equation under Neumann boundary conditions as well as for its time-fractional counterpart \cref{Eq:OriginalFisher}.

\section{Analytical Preliminaries}\label{ap}

We briefly collect the relevant preliminary results on fractional operators and the notation used throughout this work. 

\subsection{Notation}
The constant \(C\) may change from line to line, and we write \(\lesssim\) for \(\leq C\) whenever the constant \(C\) is not essential. We omit the domain \(\Omega\) when writing function spaces, e.g., \(L^2 := L^2(\Omega)\) and \(H^1 := H^1(\Omega)\). Inner products on a Banach space \(X\) are denoted by \((\cdot,\cdot)_X\), and duality pairings with its dual \(X'\) by \(\langle \cdot, \cdot \rangle_X\). In particular, \(H^{-1} = (H^1)'\). For Bochner spaces, we abbreviate the norm \(\|\cdot\|_{L^p(X)}\) for \(L^p(0,T;X)\), where \(T < \infty\) and \(p \in [1,\infty]\).

\subsection{Sonine kernels}
Sonine kernels play a central role in fractional calculus. We define the singular kernel function \( g_\alpha(t) = t^{\alpha - 1} / \Gamma(\alpha) \) for \( t \in (0, T) \) and \( \alpha > 0 \), where \( \Gamma \) denotes the Gamma function. The kernel satisfies \( g_\alpha \in L^p(0, T) \) for all \( \alpha > 1 - \tfrac{1}{p} \). Furthermore, it fulfills the convolution identity (see \cite[Theorem~2.4]{diethelm2010analysis})
\begin{equation} \label{Eq:Semigroup}
	g_\alpha * g_\beta = g_{\alpha + \beta}, 
	\qquad \forall\, \alpha, \beta \in (0, 1),
\end{equation}
known as the \emph{Sonine property}.  

We introduce the space of bounded convolutions
\[
L^p_\alpha(0,T) = \big\{ u \in L^1(0,T): \|u\|_{L^p_\alpha} := \sup_{t \in (0,T)} (\ga * |u|^p)(t) < \infty \big\}.
\]
It holds \(L^p_\alpha(0,T) \subset L^p(0,T)\) due to
\begin{equation}\begin{aligned} 
\|u\|_{L^p(0,t)}^p 
&\leq t^{1-\alpha} \int_0^t (t-s)^{\alpha-1} |u(s)|^p \,\mathrm{d}s 
\leq T^{1-\alpha} \Gamma(\alpha) (\ga * |u|^p)(t).	
\end{aligned} 
\label{Eq:KernelNorm}
\end{equation}
For a Banach space \(X\), we equip \(L^p_\alpha(0,T;X)\) with the norm \(\|\cdot\|_{L^p_\alpha(X)}\).

\subsection{Fractional derivatives}
We next recall the definitions of fractional derivatives in the sense of Riemann--Liouville and Caputo.  
The Riemann--Liouville derivative is defined by
\begin{equation}\label{Def:RL}
\Dta u = \pt(\gb * u),
\end{equation}
and the Caputo derivative is given in terms of the Riemann--Liouville derivative as \(\pta u = \Dta(u - u_0)\).  
Equivalently, if \(u\) is absolutely continuous, one has \(\pta u = \gb * \pt u\); see \cite[Lemma~3.5]{diethelm2010analysis}.  

We define the fractional Sobolev--Bochner space for \(\alpha \in (0,1)\) on \((0,T)\) with values in a Hilbert space \(H\) by
\[
W^{\alpha,p}(0,T;H) = \big\{u \in L^p(0,T;H) : \pta u \in L^{p}(0,T;H)\big\},
\]
equipped with the norm
\[
\|u\|_{W^{\alpha,p}(H)}^p = \|u\|_{L^p(H)}^p + \|\pta u \|_{L^p(H)}^p 
= \int_0^T \|u(s)\|_H^p \,\mathrm{d}s + \int_0^T \|\pta u(s)\|_H^p \,\mathrm{d}s.
\]
Analogously to the integer-order setting, there are continuous and compact embeddings for fractional Sobolev spaces. For a given Gelfand triple \(V \com H \hookrightarrow V'\), it holds \cite[Theorem~3.2]{wittbold2021bounded}
\begin{equation}\label{Eq:aubinfractional}  
W^{\alpha,p}(0,T;V') \cap L^p(0,T;V) \com L^p(0,T;H),
\quad p \in [1,\infty).
\end{equation} 
This generalizes the classical Aubin--Lions lemma to the fractional setting.

An important operational identity is the inverse relation between convolution with \(g_\alpha\) and the Caputo derivative. In fact,
\begin{align} \label{Eq:InverseConvolution}
	(\ga * \pta u)(t) &= u(t) - u_0, \qquad \forall\, u \in W^{\alpha,p}(0,T;H).
\end{align} 
Indeed,
\[
(\ga * \pta u)(t) = (\ga * \gb * \pt u)(t) = (1 * \pt u)(t)
= \int_0^t \pt u(s)\,\mathrm{d}s = u(t) - u_0,
\]
where we used \eqref{Eq:Semigroup} to conclude \(\ga * \gb = g_1 = 1\).  
Moreover, the interaction between fractional derivatives and kernel functions yields
\begin{equation} \label{Eq:DerivativeofKernel}  
\pta (\ga * u) = \Dta (\ga * u) = \pt (\gb * \ga * u) = \pt (1 * u) = u,
\quad \forall\, u \in L^1(0,T;H).
\end{equation}

The classical chain rule does not hold for fractional derivatives, but the following inequality, due to Alikhanov \cite{alikhanov2010priori}, serves as a useful substitute (see also \cite[Theorem~2.1]{vergara2008lyapunov}):
\begin{equation*}
\frac12 \pta \|u\|_H^2  \leq (u, \pta u)_H,
\quad \forall\, u \in W^{\alpha,p}(0,T;H),
\end{equation*}
for almost all \(t \in (0,T)\). This inequality can be generalized for convex functions \(G\): if \(G \in C^1\) is convex and satisfies \(G(u), G'(u) \in H\) for \(u \in H\) a.e., then \cite{fritz2021timefractional} 
\begin{equation} \label{Eq:Alikhanov}  
\pta (G(u),1)_H  \leq (G'(u), \pta u)_H,
\quad \forall\, u \in W^{\alpha,p}(0,T;H),
\end{equation}
where the case \(G(u) = \tfrac12 u^2\) recovers Alikhanov’s inequality.

\subsection{Gronwall-type inequalities}
We conclude this section with integral inequalities that will be essential for establishing a priori estimates in the analysis below. In particular, we employ Gronwall-type inequalities that accommodate convolution terms on the right-hand side. Such results are often referred to as Henry--Gronwall inequalities.  
Owing to the nonlinear nature of the Fisher--KPP model, we further require a power-type nonlinearity, leading to the so-called Bihari--Gronwall inequalities. When combined with convolution operators, these yield what are termed Henry--Gronwall--Bihari inequalities; see \cite{ouaddah2021fractional} for details.  

\begin{lemma}[Henry--Gronwall--Bihari, cf.~\cite{ouaddah2021fractional}] \label{Lem:Bihari}
Let \(c > 0\), \(u, f \in C^0([0, T];\mathbb{R}_{\geq 0})\) for some \(T>0\), and \(\psi \in C^0(\mathbb{R}_{\geq 0};\mathbb{R}_{\geq 0})\) be nondecreasing with \(\psi(0) = 0\). If
\[
u(t) \leq c + \int_0^t (t-s)^{\alpha-1} f(s) \psi(u(s)) \,\mathrm{d}s,
\]
then
\[
u(t) \leq \left[ \Psi^{-1} \left( \frac{2^q T q (p(\alpha-1)+1)}{(p(\alpha-1)+1)^q} \int_0^t f^q(s) \,\mathrm{d}s \right) \right]^{1/q}, 
\quad t \in [0, T],
\]
where
\[
\Psi(z) = \int_{2^q c^q}^z \frac{\mathrm{d}x}{(\psi(x^{1/q}))^q}, \quad 
z \geq 2^q c^q, \quad p\!\left(1 - \frac{1}{q}\right) = 1, \quad q > \frac{1}{\alpha}.
\]
\end{lemma}

In the sequel, we apply the above inequality with \(\psi(u) = u^3\), reflecting the quadratic nonlinearity of the Fisher--KPP model tested against the solution itself.

\begin{lemma}\label{Lem:BihariOur}
Let \(u \in C^0([0,T];\mathbb{R}_{\geq 0})\) for some \(T>0\), and let \(c > 0\), \(\alpha \in (0,1)\). 
If
\[
u(t) \leq c + \int_0^t (t-s)^{\alpha-1} u(s)^3 \,\mathrm{d}s,
\]  
then there exists \(T_* = T_*(c) \leq T\) such that 
\[
u(t) \leq C(T_*), \quad t \in [0, T_*].
\]  
In particular, \(c\) can be chosen sufficiently small to ensure \(T_* = T\).
\end{lemma}

\begin{proof}
We apply \cref{Lem:Bihari} with \(f(s) = 1\), \(\psi(u) = u^3\), and \(q > \frac{1}{\alpha}\) such that \(p(1 - \frac{1}{q}) = 1\), which implies \(p = \frac{q}{q-1}\). Define  
\(\beta = p(\alpha-1) + 1 = \frac{q\alpha - 1}{q-1}\).
Then
\[
\Psi(z) = \int_{2^q c^q}^z \frac{\mathrm{d}x}{(\psi(x^{1/q}))^q}
= \int_{2^q c^q}^z x^{-3}\,\mathrm{d}x
= \tfrac12 \left( \frac{1}{4^q c^{2q}} - \frac{1}{z^2} \right),
\quad z \geq 2^q c^q.
\]
From \cref{Lem:Bihari} we obtain
\[
u(t) \leq \left[ \Psi^{-1} \!\left( \frac{2^q T q \beta}{(p(\alpha-1)+1)^q} \!\int_0^t f^q(s) \,\mathrm{d}s \!\right) \right]^{1/q}, 
\quad t \in [0, T].
\]
Since \(f(s)=1\), \(\int_0^t f^q(s)\,\mathrm{d}s = t\). Solving \(\Psi(z)=A(t)\) with
\(A(t)=2^q T q \beta^{1-q} t\),
we find
\[
\frac{1}{z^2} = \frac{1}{4^q c^{2q}} - 2^{q+1} T q \beta^{1-q} t,
\quad \text{so that} \quad
z^2 = \frac{4^q c^{2q}}{1 - 2^{3q+1} T q \beta^{1-q} c^{2q} t}.
\]
Hence,
\[
u(t) \leq \frac{2c}{\left(1 - 2^{3q+1} T q \left( \tfrac{q-1}{q\alpha - 1} \right)^{q-1} c^{2q} t \right)^{1/(2q)}},
\]
for any \(t \in [0, T]\) provided the denominator is positive, i.e.,
\[
1 - 2^{3q+1} T q \beta^{1-q} c^{2q} t > 0.
\]
Thus, the bound holds for 
\(
t < T_* := \min\!\left\{ T, 1/(2^{3q+1} T q \beta^{1-q} c^{2q}) \right\}.
\)
\end{proof}

\section{Well-Posedness Analysis} \label{Sec:Analysis}

In this section, we analyze the time-fractional Fisher--KPP equation \cref{Eq:OriginalFisher}. 
In particular, we consider its convolved formulation \cref{Eq:Fisher}
and establish the existence of a weak solution up to a finite final time \( T>0 \). 

We do not attempt to prove that the solution remains bounded between \( 0 \) and \( 1 \), since the presence of the convolved nonlinearity complicates the use of weak comparison principles. Indeed, the non-convolved source term \( u(1-u) \) is more suitable for applying such principles; see, for instance, \cite{KavallarisSuzuki2018, quittner2019superlinear} for related discussions on the maximum principle for nonlocal partial differential equations. On the other hand, when considering the original formulation, the Riemann--Liouville derivative presents additional challenges for establishing a weak comparison principle, since no chain inequality analogous to that of the Caputo derivative, cf. \cref{Eq:Alikhanov}, is available for convex functions.

Instead, we establish local well-posedness in a weak setting and demonstrate the possibility of extending the solution for sufficiently small initial data. We also include, in \cref{rem2}, a remark on the global well-posedness and on the applicability of a weak comparison principle to the alternative model \cref{Eq:WrongModel}.

For simplicity, we fix $r=D=1$ in the following analysis. We also include a source term $f$, which becomes relevant in the bootstrap argument below. Thus, we consider the time-fractional Fisher--KPP equation with forcing:
\begin{equation} \label{Eq:Fisher3}
    \pta u = \Delta u + \gb*(u-u^2) + f.
\end{equation}

\subsection{Local well-posedness}

\begin{theorem}[Local well-posedness] \label{Thm:LocalExistence}
Let $\Omega \subset \R^d$, $d \le 4$, be a bounded Lipschitz domain. 
For any $u_0 \in L^2$ and $f \in L^2_\alpha(0,T;L^{4/3})$ with $T>0$, there exists 
\[
T_* = T_*\big(u_0,\|f\|_{L^2_\alpha(L^{4/3})}\big) < T
\]
such that the problem \cref{Eq:Fisher3} admits a weak solution
\[
u \in L^\infty(0,T_*;L^2) \cap L^2(0,T_*;H^1),
\quad \gb*u \in H^1(0,T_*;H^{-1}),
\]
satisfying $(\gb*(u-u_0))(0)=0$ and, for almost every $t\in[0,T_*]$,
\begin{equation}\label{Eq:Variational}
\langle \pta u,v\rangle_{H^1} + (\nabla u,\nabla v)_{L^2}
= (\gb*(u-u^2),v)_{L^2} + (f,v)_{L^2}
\quad \forall v \in H^1.
\end{equation}
\end{theorem}

We employ the Galerkin method to construct approximate solutions and derive uniform energy bounds that permit passage to the limit. Compactness results, in particular \cref{Eq:aubinfractional}, then yield a weak solution.
Similar constructions for time-fractional PDEs can be found in \cite{fritz2024analysis,fritz2021subdiffusive,fritz2021timefractional,fritz2021equivalence,fritz2024well}.
Related analytical approaches have also been applied to fractional reaction–diffusion and diffusion–wave systems in \cite{Rida2010fractional,Baranwal2012analytic,Zheng2022variable}, where compactness, energy inequalities, and fractional Sobolev embeddings play a crucial role in establishing well-posedness and qualitative behaviour.

\begin{proof}
\noindent\textbf{Step 1 (Galerkin approximation).}
Let $H_k = \mathrm{span}\{h_1,\dots,h_k\}$, where $\{h_j\}_{j=1}^\infty$ are eigenfunctions of the Laplacian satisfying
\[
(\nabla h_j,\nabla v)_{L^2} = \lambda_j (h_j,v)_{L^2}
\quad \forall v \in H^1.
\]
The set $\{h_j\}$ forms an orthonormal basis of $L^2$ and is orthogonal in $H^1$.
We seek an approximation of the form
\[
u_k(t) = \sum_{j=1}^k u_k^j(t) h_j,
\]
where $u_k^j:(0,T)\to\R$ are coefficient functions.
Given $u_0$, we define $u_{0,k} = \Pi_k u_0 \in H_k$ as its $L^2$-projection, i.e.,
$u_{0,k} = \sum_{j=1}^k u_{0,k}^j h_j$.

The Galerkin system reads: find $u_k \in H_k$ with $u_k(0)=u_{0,k}$ such that
\begin{equation}\label{Eq:FP_dis}
(\pta u_k,\zeta)_{L^2} + (\nabla u_k,\nabla \zeta)_{L^2}
= (\gb*(u_k-u_k^2),\zeta)_{L^2} + (f,\zeta)_{L^2}
\quad \forall \zeta \in H_k.
\end{equation}
Writing $\Phi(t) = (u_k^1(t),\dots,u_k^k(t))$, this system takes the form
\(\pta \Phi = f(t,\Phi)\),
where $f$ is continuous and locally Lipschitz in $\Phi$.  
By the fractional Cauchy--Lipschitz theorem \cite[Theorem~5.1]{diethelm2010analysis}, there exists a unique continuous solution $u_k \in C^0([0,T_k];H_k)$ on some interval $[0,T_k]$, $0<T_k\le T$.

\medskip
\noindent\textbf{Step 2 (Energy estimates).}
Testing \eqref{Eq:FP_dis} with $u_k$ gives
\begin{equation}\label{Eq:Tested1}
(\pta u_k,u_k)_{L^2} + \|\nabla u_k\|_{L^2}^2
= (\gb*(u_k-u_k^2),u_k)_{L^2} + (f,u_k)_{L^2}.
\end{equation}
For the first term, Alikhanov’s inequality \cref{Eq:Alikhanov} yields
\begin{equation}\label{Eq:Est1Aux1}
(\pta u_k,u_k)_{L^2} \ge \tfrac12 \pta \|u_k\|_{L^2}^2.
\end{equation}
The forcing term is estimated using Hölder, Young, and Sobolev inequalities:
\begin{equation}\label{Eq:Est1Aux1b}
(f,u_k)_{L^2} \le \|f\|_{L^{4/3}}\|u_k\|_{L^4}
\le C\|f\|_{L^{4/3}}^2 + \tfrac14\|\nabla u_k\|_{L^2}^2 + C\|u_k\|_{L^2}^2.
\end{equation}
For the nonlinear term, we apply the Young inequality:
\[
\begin{aligned}
&\int_0^t \gb(t-s) (u_k(s)-u_k^2(s),u_k(t))_{L^2}\ds \\
&\le C \!\int_0^t\! \gb(t-s)\big(\|u_k(s)\|_{L^2}^2+\|u_k(t)\|_{L^2}^2+\|u_k(s)\|_{L^3}^3+\|u_k(t)\|_{L^3}^3\big)\ds\\
&= C\big(\gb*(\|u_k\|_{L^2}^2+\|u_k\|_{L^3}^3)\big)(t)
+ \|\gb\|_{L^1}\big(\|u_k\|_{L^2}^2+\|u_k\|_{L^3}^3\big).
\end{aligned}
\]
Using the Gagliardo–Nirenberg inequality 
$$\|u_k\|_{L^3}\lesssim \|u_k\|_{L^2}^{1/2}\|\nabla u_k\|_{L^2}^{1/2},$$
see \cite[Theorem~1.24]{roubicek}, we obtain
\begin{equation}\label{Eq:Est1Aux2}
\begin{aligned}
&C\big(\gb*(\|u_k\|_{L^2}^2+\|u_k\|_{L^3}^3)\big)(t)
+ \|\gb\|_{L^1}\big(\|u_k\|_{L^2}^2+\|u_k\|_{L^3}^3\big)\\
&\le C\!\Big(\gb*\big(\|u_k\|_{L^2}^2+\|u_k\|_{L^2}^{3/2}\|\nabla u_k\|_{L^2}^{3/2}\big)\Big)(t)
\\ &\quad+ \|\gb\|_{L^1}\big(\|u_k\|_{L^2}^2+\|u_k\|_{L^2}^{3/2}\|\nabla u_k\|_{L^2}^{3/2}\big)\\
&\le C\!\Big(\gb*\big(\|u_k\|_{L^2}^2+\|u_k\|_{L^2}^6\big)\Big)(t)
+ C\big(\|u_k\|_{L^2}^2+\|u_k\|_{L^2}^6\big)
\\ &\quad+ \varepsilon (\gb*\|\nabla u_k\|_{L^2}^2)(t)
+ \tfrac14 \|\nabla u_k\|_{L^2}^2,
\end{aligned}
\end{equation}
where the last inequality follows from the $\varepsilon$–Young inequality.

Substituting \eqref{Eq:Est1Aux1}–\eqref{Eq:Est1Aux2} into \eqref{Eq:Tested1}, we obtain
\[
\begin{aligned}
&\tfrac12 \pta \|u_k\|_{L^2}^2 + \tfrac12 \|\nabla u_k\|_{L^2}^2
\\&\le C\big(\gb*(\|u_k\|_{L^2}^2+\|u_k\|_{L^2}^6)\big)(t)
+ C(\|u_k\|_{L^2}^2+\|u_k\|_{L^2}^6)
\\ &\quad + \varepsilon(\gb*\|\nabla u_k\|_{L^2}^2)(t)
+ C\|f\|_{L^{4/3}}^2.
\end{aligned}
\]
Convolving with $\ga$ and using $\ga*\gb=1$ and $\ga*\pta w=w-w(0)$ (see \cref{Eq:Semigroup,Eq:InverseConvolution}) yields
\begin{equation}\label{Eq:Galerkin2}
\begin{aligned}
&\|u_k(t)\|_{L^2}^2 + (\ga*\|\nabla u_k\|_{L^2}^2)(t)
\\ &\le \|u_{0,k}\|_{L^2}^2 + C(\ga*\|f\|_{L^{4/3}}^2)(T)\\
&\quad + C\!\int_0^t (1+\ga(t-s))(\|u_k(s)\|_{L^2}^2+\|u_k(s)\|_{L^2}^6)\ds\\
&\quad + 2\varepsilon\!\int_0^t \|\nabla u_k(s)\|_{L^2}^2\ds.
\end{aligned}
\end{equation}

Because $\|u_{0,k}\|_{L^2}\le\|u_0\|_{L^2}$ and \cref{Eq:KernelNorm} implies 
$$(\ga*\|\nabla u_k\|_{L^2}^2)(t)\gtrsim \|\nabla u_k\|_{L^2(0,t;L^2)}^2,$$  
we can choose $\varepsilon$ small enough to absorb the last term.
Hence,
\[
\|u_k(t)\|_{L^2}^2
\le \|u_0\|_{L^2}^2 + \|f\|_{L^2_\alpha(L^{4/3})}^2
+ C\!\int_0^t \ga(t-s)\big(\|u_k(s)\|_{L^2}^2+\|u_k(s)\|_{L^2}^6\big)\ds.
\]
Applying the Henry–Bihari–Gronwall lemma (\cref{Lem:BihariOur}) yields
\begin{equation}\label{Eq:FinalEst1}
\|u_k(t)\|_{L^2}^2
\le C(T_*)\big(\|u_0\|_{L^2}^2+\|f\|_{L^2_\alpha(L^{4/3})}^2\big),
\quad t\in[0,T_*],
\end{equation}
for some sufficiently small $T_*=T_*(\|u_0\|_{L^2},\|f\|_{L^2_\alpha(L^{4/3})})$.  
This uniform bound implies that $u_k$ is also bounded in $L^2(0,T_*;H^1)$. \medskip

\noindent\textbf{Step 3 (Weak and strong convergence).} 
From the uniform bound \eqref{Eq:FinalEst1}, we infer that $(u_k)$ is bounded in both 
$L^\infty(0,T_*;L^2)$ and $L^2(0,T_*;H^1)$. 
Hence, there exists a limit function $u$ (up to a subsequence, not relabeled) such that
\begin{equation}\label{Eq:Convergence2a}
\begin{aligned}
u_k &\rightharpoonup u &&\text{ in } L^2(0,T_*;H^1),\\
u_k &\stackrel{*}{\rightharpoonup} u &&\text{ in } L^\infty(0,T_*;L^2),
\end{aligned}
\end{equation}
as $k \to \infty$. 

Since the equation contains the nonlinear term $u_k^2$, we require a strong convergence result for $u_k$. 
To this end, we test the discrete system \eqref{Eq:FP_dis} with $\Pi_k\zeta$, where $\zeta\in L^2(0,T;H^1)$ is arbitrary. 
Using the Young convolution inequality, we obtain
\begin{equation}\label{Eq:DerivBound}
\begin{aligned}
&(\pta u_k,\Pi_k\zeta)_{L^2(L^2)} 
\\ &= (f,\Pi_k\zeta)_{L^2(L^2)} 
  + (\gb*(u_k-u_k^2),\Pi_k\zeta)_{L^2(L^2)} 
  - (\nabla u_k,\nabla \Pi_k\zeta)_{L^2(L^2)}\\
&\le \Big(
\|f\|_{L^2(L^{4/3})}
+ \|\gb*(u_k-u_k^2)\|_{L^2(L^{4/3})}
+ \|\nabla u_k\|_{L^2(L^2)}
\Big)\|\Pi_k\zeta\|_{L^2(H^1)}\\
&\le \Big(
\|f\|_{L^2(L^{4/3})}
+ \|\gb\|_{L^1}\big(
\|u_k\|_{L^2(L^2)} 
+ \|u_k\|_{L^\infty(L^2)}\|u_k\|_{L^2(L^4)}\big)
\\&\quad + \|\nabla u_k\|_{L^2(L^2)}
\Big)\|\zeta\|_{L^2(H^1)}\\
&\le C\,\|\zeta\|_{L^2(H^1)},
\end{aligned}
\end{equation}
which implies that $\pta u_k$ is bounded in $L^2(0,T;H^{-1})$ uniformly in $k$. 

By the compact embedding 
\cite[cf.~Theorem 3.2]{wittbold2021bounded},
\[
L^2(0,T;H^1)\cap H^\alpha(0,T;H^{-1})
\com L^2(0,T;L^2),
\]
we deduce the strong convergence
\begin{equation}\label{Eq:StrongConv1}
u_k \to u \quad\text{strongly in } L^2(0,T;L^2).
\end{equation}

\medskip
\noindent\textbf{Step 4 (Limit passage).}
We now pass to the limit $k\to\infty$ in the time-integrated Galerkin formulation \eqref{Eq:FP_dis}. 
Using the convergences derived above, we show that the weak limit $u$ satisfies the variational identity \eqref{Eq:Variational}. 

For any $\zeta\in H_k$ and $\eta\in C_c^\infty(0,T)$, consider the time-integrated form:
\[
\begin{aligned}
&\int_0^T\!\Big(
\langle \pta u_k,\zeta\rangle_{H^1}
+ (\nabla u_k,\nabla\zeta)_{L^2}
- (f,\zeta)_{L^2}\Big)\eta(t)\,dt\\
&= \int_0^T\!\!\int_0^t \gb(t-s)
\big(u_k(s)(1-u_k(s)),\zeta\big)_{L^2}\ds\,\eta(t)\,dt.
\end{aligned}
\]
All linear terms converge directly by weak convergence.
For the nonlinear term, note that $u_k\zeta\rightharpoonup u\zeta$ weakly in $L^2(0,T;L^2)$ and $1-u_k\to1-u$ strongly in $L^2(0,T;L^2)$, hence $$u_k(1-u_k)\zeta\rightharpoonup u(1-u)\zeta \text{ weakly in } L^1(0,T;L^1).$$ 
Since convolution with $\gb$ is linear and continuous from $L^1$ to $L^1$, it is weak-to-weak continuous. 
Therefore, $$\gb*(u_k(1-u_k)\zeta)\rightharpoonup \gb*(u(1-u)\zeta) \text{ weakly in } L^1(0,T;L^1).$$ 

Finally, because $\bigcup_k H_k$ is dense in $H^1$, the limit $u$ satisfies the variational formulation \eqref{Eq:Variational} for all $v\in H^1$, and hence is a weak solution to \cref{Eq:Fisher3}.

\medskip
\noindent\textbf{Step 5 (Initial condition).}
Since $\pta u_k = \pt(\gb*(u_k-u_0))$ is bounded in $L^2(0,T;H^{-1})$ (by \eqref{Eq:DerivBound}), the Aubin–Lions lemma \cite[Theorem II.5.16]{boyer2012mathematical} implies that
\[
\gb*(u_k-u_k(0)) \to \gb*(u-u_0)
\quad\text{in } L^2(0,T;L^2)\cap C([0,T];H^{-1}).
\]
Using the Young convolution inequality to justify boundedness, we evaluate at $t=0$ and obtain
\[
0 = (\gb*(u_k-u_k(0)))(0)
\longrightarrow (\gb*(u-u_0))(0)
\quad\text{in } H^{-1}(\Omega),
\]
hence $(\gb*(u-u_0))(0)=0$. 
This verifies the initial condition and concludes the proof.
\end{proof}

\subsection{Global well-posedness for small data}

We now establish global existence and decay for small initial data. 
The key idea is that, for sufficiently small $\|u_0\|_{L^2}$, the dissipative effects dominate the nonlinear growth, which allows the local solution to be extended indefinitely in time. 
Before presenting the main theorem, we first prove an auxiliary lemma that plays a crucial role in the continuation argument.

\begin{lemma}[Regularity of the history force] \label{Lem:FRegularityL43}
Let \(u \in L^\infty(0,T_*;L^2) \cap L^2(0,T_*;H^1)\) be the solution obtained in \cref{Thm:LocalExistence}, satisfying 
\(\sup_{t \in (0,T_*)}\|u(t)\|_{L^2}^2 \le 2\epsilon\) 
for some $\epsilon>0$. 
Define the history force
\[
F(t) = \int_0^{T_*} g_\alpha(t + T_* - s)\,(u(s) - u(s)^2)\,\dd s.
\]
Then, for any \(T > 0\), we have \(F \in L^2_\alpha(0,T;L^{4/3})\).
\end{lemma}

\begin{proof}
We first estimate pointwise in time:
\begin{align*}
\|F(t)\|_{L^{4/3}}
&\le \int_0^{T_*} g_\alpha(t + T_* - s)\,\|u(s) - u(s)^2\|_{L^{4/3}}\,\dd s\\
&\le \int_0^{T_*} g_\alpha(t + T_* - s)
\big(\|u(s)\|_{L^{4/3}} + \|u(s)^2\|_{L^{4/3}}\big)\,\dd s.
\end{align*}
Using the uniform bound $\|u(s)\|_{L^2}^2 \le 2\epsilon$, we have
\[
\|u(s)\|_{L^{4/3}} \le C \|u(s)\|_{L^2} \le C\sqrt{2\epsilon}.
\]
For the quadratic term, Hölder and the Gagliardo--Nirenberg inequality yield
\[
\|u(s)^2\|_{L^{4/3}} = \|u(s)\|_{L^{8/3}}^2
   \le C\,\|u(s)\|_{L^2}\|u(s)\|_{L^4}
   \le C\sqrt{2\epsilon}\,\|u(s)\|_{L^4}.
\]
Thus,
\[
\|u(s)-u(s)^2\|_{L^{4/3}}
   \le C\sqrt{\epsilon}\,\big(1+\|u(s)\|_{L^4}\big),
\]
and consequently,
\[
\|F(t)\|_{L^{4/3}}
   \le C\sqrt{\epsilon}\int_0^{T_*} g_\alpha(t + T_* - s)
   \big(1+\|u(s)\|_{L^4}\big)\,\dd s.
\]
Introduce the auxiliary functions
\[
G(t) = 
\begin{cases}
g_\alpha(t), & t \ge 0,\\
0, & t < 0,
\end{cases}
\qquad
H(s) =
\begin{cases}
1+\|u(s)\|_{L^4}, & s \in [0,T_*],\\
0, & \text{otherwise}.
\end{cases}
\]
Then $\|F(t)\|_{L^{4/3}} \le C\sqrt{\epsilon}\,(G*H)(t+T_*)$, and hence
\[
\|F\|_{L^2(0,T;L^{4/3})}
   \le C\sqrt{\epsilon}\,\|g_\alpha\|_{L^1(0,\infty)}\,\|H\|_{L^2(0,T_*)}.
\]
Since $H(s)=1+\|u(s)\|_{L^4}$ and $u\in L^2(0,T_*;H^1)\hookrightarrow L^2(0,T_*;L^4)$, it follows that $\|H\|_{L^2(0,T_*)}<\infty$, and thus $F\in L^2(0,T;L^{4/3})$.

Next, we show that $F\in L^2_\alpha(0,T;L^{4/3})$.  
Indeed,
\begin{align*}
(\ga * \|F\|_{L^{4/3}}^2)(t)
&= \int_0^t \ga(t-\tau)\|F(\tau)\|_{L^{4/3}}^2\,\dd\tau\\
&\le C\epsilon\int_0^t \ga(t-\tau)\,[(G*H)(\tau+T_*)]^2\,\dd\tau.
\end{align*}
With the change of variables $\sigma=\tau+T_*$, we obtain for $t\in(0,T)$
\[
(\ga * \|F\|_{L^{4/3}}^2)(t)
= C\epsilon\int_{T_*}^{t+T_*}\ga(t+T_*-\sigma)\,[(G*H)(\sigma)]^2\,\dd\sigma.
\]
Since $G*H\in L^2(T_*,T+T_*)$, the integral is finite for all $t\in(0,T)$, proving the claim.
\end{proof}

\begin{theorem}[Global well-posedness for small data] \label{Thm:GlobalSmallData}
Let $\Omega\subset\R^d$, $d\le4$, be a bounded Lipschitz domain. 
There exists $\epsilon>0$ such that if $u_0\in L^2$ with $\|u_0\|_{L^2}\le\epsilon$, then the time-fractional Fisher--KPP equation \cref{Eq:Fisher} admits a unique global weak solution
\[
u \in L^\infty(0,\infty;L^2)\cap L^2(0,\infty;H^1),
\qquad g_\alpha*u\in H^1(0,\infty;H^{-1}),
\]
satisfying $u(0)=u_0$ weakly in $L^2$, and for almost every $t>0$,
\[
\langle \partial_t^\alpha u,v\rangle_{H^1}
+ (\nabla u,\nabla v)_{L^2}
= (g_\alpha*(u-u^2),v)_{L^2}
\quad \forall v\in H^1.
\]
\end{theorem}

\begin{proof}
We combine the local well-posedness result \cref{Thm:LocalExistence} with a continuation (bootstrap) argument exploiting the smallness of the initial data.

\medskip
\noindent\textbf{Step 1 (Local solution and bootstrap).}
From \cref{Thm:LocalExistence} with $f\equiv0$, there exists $T_*=T_*(\|u_0\|_{L^2})>0$ and a solution
\[
u \in L^\infty(0,T_*;L^2)\cap L^2(0,T_*;H^1),
\qquad g_\alpha*u\in H^1(0,T_*;H^{-1}).
\]
Testing the variational formulation with $v=u$ and applying Alikhanov’s inequality \cref{Eq:Alikhanov}, we derive
\begin{equation}\label{Eq:EnergyLocal}
\frac12\,\partial_t^\alpha\|u(t)\|_{L^2}^2+\|\nabla u(t)\|_{L^2}^2
\le C\!\int_0^t g_\alpha(t-s)\big(\|u(s)\|_{L^2}^2+\|u(s)\|_{L^2}^6\big)\,\dd s.
\end{equation}
Setting $y(t):=\|u(t)\|_{L^2}^2$, this gives
\[
y(t)\le y(0)+C\!\int_0^t\ga(t-s)\big(y(s)+y(s)^3\big)\,\dd s.
\]
By the Henry--Bihari--Gronwall inequality (\cref{Lem:BihariOur}), there exists $T_*>0$ such that, if $\|u_0\|_{L^2}^2=y(0)\le\epsilon$ is sufficiently small,
\[
y(t)\le 2\epsilon\quad\forall\,t\in[0,T_*].
\]
Hence $\|u(t)\|_{L^2}$ remains small on the interval $[0,T_*]$.

\medskip
\noindent\textbf{Step 2 (Maximal time of existence).}
Define
\[
T_{\max}:=\sup\big\{T>0:\|u(t)\|_{L^2}^2\le2\epsilon~\forall\,t\in[0,T]\big\}.
\]
We will show that $T_{\max}=\infty$. 
Suppose, to the contrary, that $T_{\max}<\infty$. 
Then $\|u(T_{\max})\|_{L^2}^2\le2\epsilon$. 
We now restart the system at time $t=T_{\max}$.

\medskip
\noindent\textbf{Step 3 (Shifted problem and history term).}
Define $v(t):=u(t+T_{\max})$. 
Then $v$ satisfies
\[
\partial_t^\alpha v-\Delta v=g_\alpha*(v-v^2)+F(t),
\qquad v(0)=u(T_{\max}),
\]
where
\[
F(t)=\int_0^{T_{\max}} g_\alpha(t+T_{\max}-s)\,(u(s)-u(s)^2)\,\dd s.
\]
By \cref{Lem:FRegularityL43}, $F\in L^2_\alpha(0,\delta;L^{4/3})$ for any $\delta>0$, 
so the shifted problem fits the setting of \cref{Thm:LocalExistence}.

\medskip
\noindent\textbf{Step 4 (Local solution for the shifted problem).}
Applying \cref{Thm:LocalExistence} to $v$ with initial condition $v(0)=u(T_{\max})$ and forcing $f=F$, we obtain $\delta>0$ such that
\[
v\in L^\infty(0,\delta;L^2)\cap L^2(0,\delta;H^1),
\quad g_\alpha*v\in H^1(0,\delta;H^{-1}),
\]
and
\[
\|v(t)\|_{L^2}^2
\le \|v(0)\|_{L^2}^2
+ C\!\int_0^t g_\alpha(t-s)\big(\|F(s)\|_{L^{4/3}}^2+\|v(s)\|_{L^2}^2+\|v(s)\|_{L^2}^6\big)\,\dd s.
\]
Since $\|v(0)\|_{L^2}^2\le2\epsilon$ and $F\in L^2_\alpha(0,\delta;L^{4/3})$, 
choosing $\epsilon$ sufficiently small ensures $\|v(t)\|_{L^2}^2\le2\epsilon$ on $[0,\delta]$. 
Hence $u$ extends beyond $T_{\max}$, a contradiction. 
Therefore $T_{\max}=\infty$ and $\|u(t)\|_{L^2}^2\le2\epsilon$ for all $t\ge0$.

\medskip
\noindent\textbf{Step 5 (Global existence and uniqueness).}
Repeating the extension argument iteratively yields a global solution
\[
u\in L^\infty(0,\infty;L^2)\cap L^2(0,\infty;H^1),
\qquad g_\alpha*u\in H^1_{\mathrm{loc}}(0,\infty;H^{-1}),
\]
satisfying $\|u(t)\|_{L^2}^2\le2\epsilon$ for all $t\ge0$. 
Uniqueness follows from the local uniqueness of \cref{Thm:LocalExistence} and the continuation process.
\end{proof}

\begin{remark}
With homogeneous Dirichlet boundary conditions, the solution satisfies a fractional energy decay. 
By the Poincaré inequality $\|\nabla u\|_{L^2}^2\ge\lambda_1\|u\|_{L^2}^2$, \eqref{Eq:EnergyLocal} implies
\[
y(t)\le y(0)+\Big(C-\frac{\lambda_1}{2}\Big)\!\int_0^t\ga(t-s)y(s)\,\dd s.
\]
By the fractional Grönwall inequality \cite[Lemma~6.19]{diethelm2010analysis}, if $\lambda_1/2>C$, then
\[
\|u(t)\|_{L^2}^2 \le 2\|u_0\|_{L^2}^2E_\alpha(-\mu t^\alpha), \qquad 
\mu=\frac{\lambda_1}{2}-C>0,
\]
where $E_\alpha$ denotes the Mittag–Leffler function. 
Thus, $\|u(t)\|_{L^2}\to0$ as $t\to\infty$. 
Without Dirichlet data, uniform bounds and dissipation of 
$\int_0^t\ga(t-s)\|\nabla u(s)\|_{L^2}^2\,\dd s$ remain valid, 
but decay may fail due to the non-decaying mean mode.
\end{remark}

\begin{remark}\label{rem2}
The smallness assumption on $\|u_0\|_{L^2}$ is crucial for the bootstrap argument. 
For larger data, finite-time blow-up cannot be excluded. 
However, for the alternative model \cref{Eq:WrongModel}, global well-posedness holds for initial data $u_0(x)\in[0,1]$. 
Here, the logistic term is non-convolved, and a weak comparison principle applies as in \cite{gunzburger2005modeling}. 
Following \cite{fritz2025wellposedness}, we test the variational form with $-[-u]_+$, where $[u]_+=\max\{0,u\}$, obtaining via Alikhanov’s inequality
\[
\tfrac12\,\pta\|[-u]_+\|_{L^2}^2+\|\nabla[-u]_+\|_{L^2}^2
\le(1+[-u]_+,[-u]_+^2)_{L^2}.
\]
Using Hölder and the Gagliardo--Nirenberg inequality, and then convolving with $\ga$, one finds that $\|[-u]_+\|_{L^2}=0$ locally (by a Bihari--Gronwall argument). 
This ensures global nonnegativity and global existence whenever $u_0\in[0,1]$.
\end{remark}

\section{Numerical experiments} \label{Sec:Numerics}

Numerical treatment of time-fractional partial differential equations poses several challenges due to the presence of nonlocal operators and long-range memory effects.
We adapt techniques from the fractional-diffusion literature \cite{diethelm2020good,jin2019numerical} and build on our previous work on time-fractional Fokker–Planck equations \cite{fritz2024analysis}.
Key numerical references include the convolution-quadrature method \cite{Lubich86,Lubich88,dolz2021fast}, graded temporal meshes \cite{pinto2017numerical,mustapha2022second}, and kernel-compression approaches \cite{baffet2017kernel,khristenko2023solving}.
Further developments of numerical schemes for fractional PDEs—such as Runge–Kutta convolution quadrature and fast compact finite-difference or finite-element solvers—can be found in \cite{Zhang2020RungeKutta,Yin2024Hadamard,BetancurHerrera2020Caputo,Feng2021anomalous}, which provide additional validation of the stability and convergence properties of convolution-based discretizations used here.

\subsection{Temporal and spatial discretization}

We employ a graded temporal grid
\[
0=t_0<t_1<\dots<t_N=T, 
\qquad t_n=\Big(\frac{n}{N}\Big)^\gamma T,
\]
with grading parameter $\gamma\ge1$ and local step size $\Delta t_n=t_n-t_{n-1}$. 
The spatial domain $\Omega\subset\R^2$ is discretized using conforming $P_1$ finite elements with mesh size $h$, and we denote by $V_h\subset H^1$ the corresponding finite element space.

For $n\ge1$ we introduce the quadrature weights
\[
b_j^{(n)}:=\frac{(t_n-t_j)^{1-\alpha}-(t_n-t_{j+1})^{1-\alpha}}{\Gamma(2-\alpha)}, 
\qquad j=0,\dots,n-1,
\]
which represent piecewise-constant quadrature of the kernel $g_{1-\alpha}(t)=t^{-\alpha}/\Gamma(1-\alpha)$ over the interval $[t_j,t_{j+1}]$. 

These coefficients appear in two contexts:
\begin{itemize}
    \item[(i)] The graded $L^1$ scheme \cite{lin2007finite,jin2019numerical,fritz2024well} for the Caputo derivative is expressed as
    \[
    \partial_t^\alpha u(t_n)\approx 
    \sum_{k=1}^n a_{n-k}^{(n)}(u^k-u^{k-1})
    :=\sum_{k=1}^n\frac{b_{k-1}^{(n)}}{\Delta t_k}(u^k-u^{k-1}),
    \]
    so that the coefficients $a_{n-k}^{(n)}$ are rescaled quadrature weights $b_{k-1}^{(n)}$.
    \item[(ii)] For a generic function $f$, the convolution $(\gb*f)(t_n)$ is approximated by
    \[
    (\gb*f)(t_n)\approx \sum_{j=0}^{n-1}b_j^{(n)}f^j,
    \]
    where $f^j$ denotes the piecewise-constant value of $f$ on $[t_j,t_{j+1})$.
\end{itemize}

\subsection{Fully discrete formulation}

The fully discrete scheme for the time-fractional Fisher–KPP equation \cref{Eq:Fisher}, as analyzed in \cref{Sec:Analysis}, reads as follows:  
find $u_h^n\in V_h$ such that, for all $\zeta_h\in V_h$,
\begin{equation}\label{eq:discrete_A}
\begin{aligned}
&a_0^{(n)}(u_h^n,\zeta_h)_{L^2}
+ D(\nabla u_h^n,\nabla\zeta_h)_{L^2}
- r b_n^{(n)}(u_h^n(1-u_h^n),\zeta_h)_{L^2}\\
&= a_0^{(n)}(u_h^{n-1},\zeta_h)_{L^2}
- \sum_{k=1}^{n-1}a_{n-k}^{(n)}(u_h^k-u_h^{k-1},\zeta_h)_{L^2}
+ r\sum_{j=0}^{n-1}b_j^{(n)}(u_h^j(1-u_h^j),\zeta_h)_{L^2}.
\end{aligned}
\end{equation}

For the alternative (Caputo-in-time) formulation \cref{Eq:WrongModel}, the discrete system becomes:
find $u_h^n\in V_h$ such that, for all $\zeta_h\in V_h$,
\[
\begin{aligned}
&a_0^{(n)}(u_h^n,\zeta_h)_{L^2}
+ D(\nabla u_h^n,\nabla\zeta_h)_{L^2}
- r(u_h^n(1-u_h^n),\zeta_h)_{L^2}\\
&= a_0^{(n)}(u_h^{n-1},\zeta_h)_{L^2}
- \sum_{k=1}^{n-1}a_{n-k}^{(n)}(u_h^k-u_h^{k-1},\zeta_h)_{L^2}.
\end{aligned}
\]

\subsection{Implementation and setup}

All computations were carried out using the Firedrake framework \cite{rathgeber2016firedrake}.  
We set $\Omega=(-1,1)^2$ with mesh size $h=2^{-7}$ and final time $T=5$, using a temporal grading $\gamma=2$. 
Model parameters were fixed as $D=10^{-3}$ and $r=5$.

Three types of initial data were considered:
\begin{enumerate} \itemsep0em
    \item a single circular patch,
    \item four separated circles,
    \item a complex “blob-shaped’’ region defined by a smooth level-set function
    \[
    \begin{aligned}
    s(x,y)&=(\sin(6(x-0.6)+2(y-0.5))+1)\,(7(x-0.6)-0.2)^2\\
    &\quad+(\sin(-8(x-0.6)+10(y-0.5))+1.1)\,(9(y-0.5)+0.1)^2-1,
    \end{aligned}
    \]
    restricted to $0.05<x<0.9$ and $0.1<y<0.85$.
\end{enumerate}
The initial condition is given by the smoothed indicator
\[
u_0(x,y)=\frac{1}{2}\bigl(1-\tanh(s(x,y)/\varepsilon)\bigr),
\]
with $\varepsilon=10h$, producing a smooth transition layer of width $\approx10h$ between values $0$ and $1$. 
The three initial geometries are depicted in \cref{fig:initial}.

\begin{figure}[htb!]
    \centering
    \includegraphics[width=.9\textwidth,page=1]{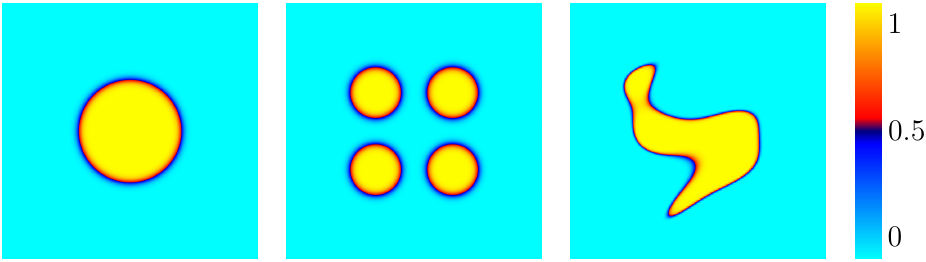}
    \caption{Initial conditions: (left) single circular patch, (middle) four separated circles, (right) irregular ``blob'' geometry.}
    \label{fig:initial}
\end{figure}

\subsection{Numerical results}

\Cref{fig:mass} shows the temporal evolution of the total mass for various fractional orders and initial geometries.  
For all cases, the physically consistent formulation \cref{Eq:Fisher} produces a smoother, monotonic mass evolution approaching the classical ($\alpha=1$) limit, 
while the Caputo-in-time formulation \cref{Eq:WrongModel} displays transient overshoot and faster apparent front propagation.  
In biological terms, the consistent model represents subdiffusive transport with delayed spread due to crowding or trapping, whereas the Caputo-type model exaggerates effective diffusivity, yielding unphysical long-range growth effects.

\begin{figure}[htb!]
    \centering
    \begin{subfigure}[c]{0.49\textwidth}
\includegraphics[width=.95\textwidth,page=2]{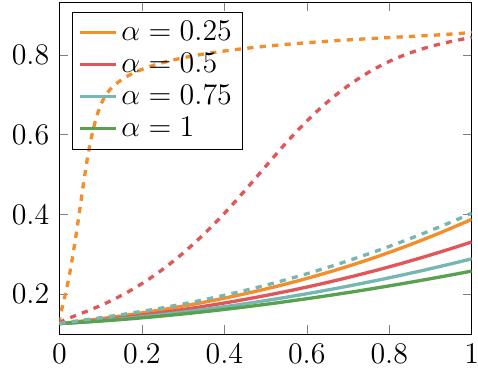}
        \caption{Mass evolution for varying $\alpha$.}
    \end{subfigure}
    \begin{subfigure}[c]{0.49\textwidth}
        \includegraphics[width=.95\textwidth,page=3]{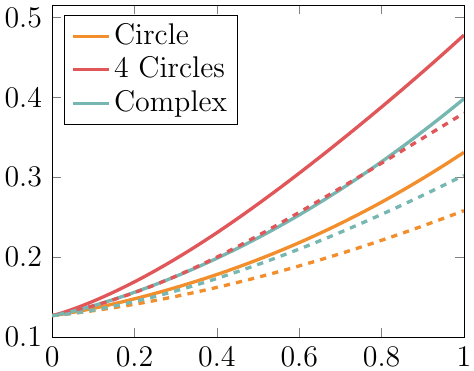}
        \caption{Mass evolution for different initial geometries.}
    \end{subfigure}
    \caption{Total mass evolution under both models:
    solid lines correspond to \cref{Eq:Fisher}, dashed lines to \cref{Eq:WrongModel}.}
    \label{fig:mass}
\end{figure} 

\noindent\textbf{Morphological effects.}
Differences in initial geometry (\cref{fig:mass}b) highlight the coupling between interface length and growth rate. 
Configurations with extended interfaces (e.g., multiple circles) exhibit faster early mass accumulation due to their larger effective perimeter, 
while compact morphologies expand more slowly but reach similar steady states.  
This interface-limited proliferation aligns with the classical theory of diffusive tumor fronts \cite{AraujoMcElwain2004,RooseChapmanMaini2007,DrasdoHoehme2005,Frieboes2006,WiseLowengrubFrieboesCristini2008}.

\smallskip
\noindent\textbf{Model comparison.}
Spatial snapshots in \cref{fig:models_alpha05} compare both models for $\alpha=\tfrac12$. 
The consistent model maintains compact aggregates with slower, well-defined interface motion, 
while the Caputo-in-time variant develops diffuse boundaries and reduced saturation ($u\approx1$) behind the front.  
Consequently, the Caputo model reaches domain boundaries faster, reflecting its artificially enhanced front velocity.

\begin{figure}[htb!]
    \centering
    \includegraphics[width=.9\textwidth,page=4]{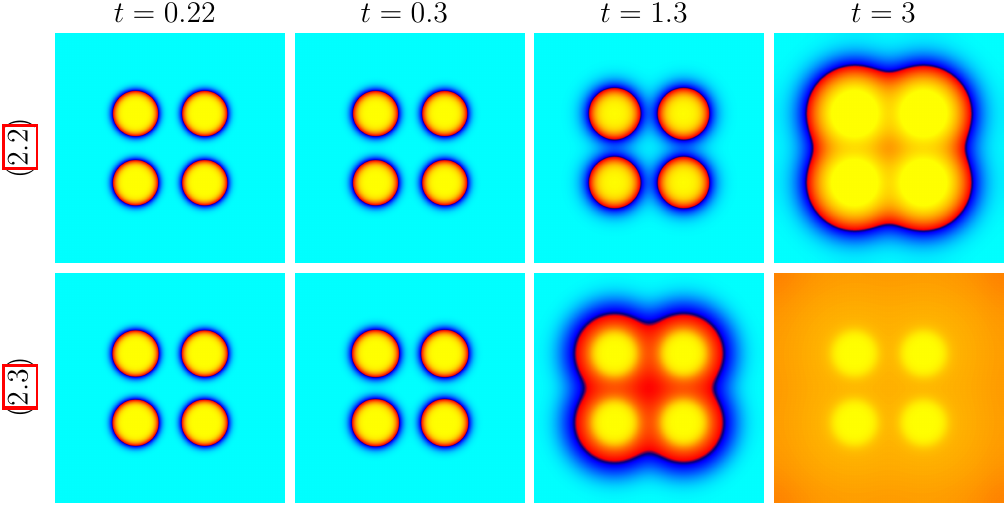}
\vspace{-.3cm}\caption{Comparison of consistent \cref{Eq:Fisher} and Caputo-in-time \cref{Eq:WrongModel} for $\alpha=\tfrac12$ at different times.}
\label{fig:models_alpha05}
\end{figure}

\smallskip
\noindent\textbf{Dependence on fractional order.}
\Cref{fig:alpha_sweep} shows results for the complex initial geometry and varying $\alpha$. 
Decreasing $\alpha$ enhances memory effects, producing thicker interfaces and delayed front motion.  
Although fronts appear broader, this arises from persistent subdiffusive dynamics rather than acceleration.  
For smaller $\alpha$, spatial gradients flatten and the interior saturation decreases, consistent with stronger anomalous trapping.

\begin{figure}[htb!]
    \centering
    \includegraphics[width=.9\textwidth,page=5]{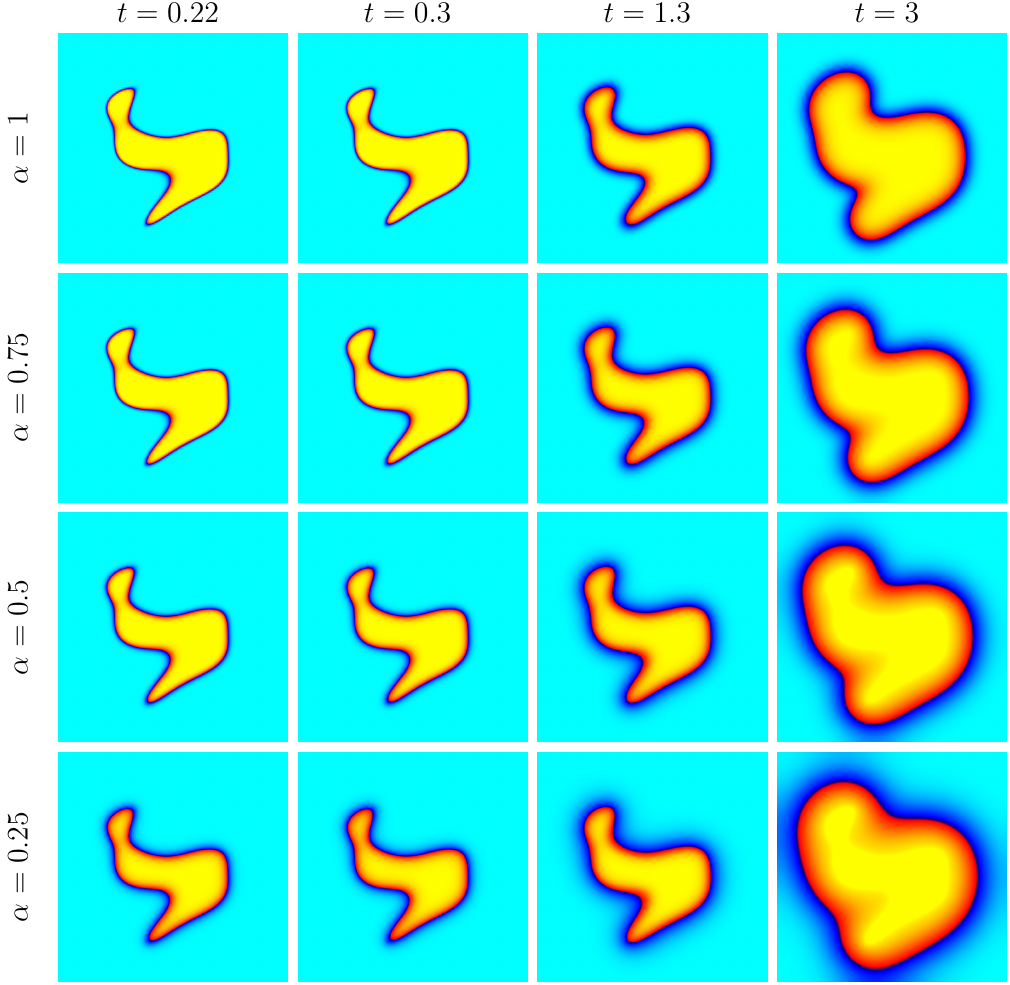}
\vspace{-.3cm}\caption{Evolution of the complex initial condition for fractional orders $\alpha\in\{\tfrac14,\tfrac12,\tfrac34,1\}$.}
\label{fig:alpha_sweep}
\end{figure}

\section{Conclusions}

We have presented a rigorous and computational study of a physically consistent time-fractional Fisher--KPP equation.
Unlike the commonly adopted Caputo-in-time formulation, the present model employs a Riemann--Liouville derivative acting within the diffusion term, in accordance with first-principles derivations from statistical physics \cite{angstmann2020time,heinsalu2007use,weron2008modeling}. 
This distinction, although subtle at the analytical level, leads to qualitatively different dynamical behaviour, both mathematically and numerically.

From an analytical viewpoint, we established local well-posedness via a Galerkin approximation combined with fractional Gronwall-type estimates, and extended the result globally for sufficiently small initial data. 
Our proof relies on a bootstrap argument that controls the nonlinear growth through the dissipative structure of the equation. 
Numerical simulations, performed with graded convolution-quadrature schemes and finite-element spatial discretization, confirmed that the physically justified formulation yields slower, subdiffusive front propagation and realistic saturation dynamics, in contrast to the unphysical acceleration observed in the Caputo-in-time variant.

A central open question concerns the global well-posedness of the consistent fractional Fisher--KPP equation for arbitrary nonnegative initial data \(u_0 \in [0,1]\). 
While smallness of \(\|u_0\|_{L^2}\) ensures global existence, extending this result to general bounded data remains unresolved due to the absence of a suitable weak comparison principle for the Riemann--Liouville formulation. 
Addressing this problem would close a fundamental gap between the mathematically and physically justified models.

Another promising research direction lies in inverse problems and parameter identification, namely the joint estimation of the fractional order and kinetic coefficients from partial or noisy observations.
Recent studies on fractional inverse  \cite{Yan2023inverse} provide an analytical and computational foundation for this line of work and suggest that coupling such estimation with biologically informed Fisher–KPP dynamics could yield new quantitative insights into anomalous tumor growth and subdiffusive transport.

{\small	
	\bibliography{literature.bib}
	\bibliographystyle{abbrv} }

\end{document}